\documentclass[12pt]{amsart}

\usepackage{amsthm, amsmath, amssymb, url, verbatim,xcolor}
\usepackage[margin=1.25in]{geometry}
{\providecommand{\noopsort}[1]{} %year = "unpublished manuscript\setbox0=\hbox{2003}"

\frenchspacing

\theoremstyle{plain}
\newtheorem{theorem}{Theorem}[section]%\newtheorem*{nonumbertheorem}{Theorem}
\newtheorem*{nonumbercorollary}{Corollary}
\newtheorem{lemma}[theorem]{Lemma}

\newtheorem*{maintheorem}{Main Theorem}
\theoremstyle{definition}

\theoremstyle{remark}

\numberwithin{equation}{section}

\newcommand{\ga}{\gamma}\newcommand{\de}{\delta}

\newcommand{\N}{\mathbb{N}}
\newcommand{\R}{\mathbb{R}}

\newcommand{\s}{\mathbb{S}}

\newcommand{\of}[1]{\left(#1\right)}

\newcommand{\floor}[1]{\left\lfloor #1 \right\rfloor}
\newcommand{\pfrac}[2]{\left(\frac{#1}{#2}\right)}
\newcommand{\st}{~|~}

\title[Characterizations of the round $2$--sphere in terms of closed geodesics]{Characterizations of the round two-dimensional sphere in terms of closed geodesics}
\author{Lee Kennard}
\address{Department of Mathematics, University of California, Santa Barbara, CA 93106}
\email{kennard@math.ucsb.edu}
\author{Jordan Rainone}
\address{Department of Mathematics, SISSA International School for Advanced Studies, Via Bonomea, 265, 34136 Trieste, Italy}
\email{jrainone@sissa.it}
\date{\today}

\begin{document}
\begin{abstract}
The question of whether a closed Riemannian manifold has infinitely many geometrically distinct closed geodesics has a long history. Though unsolved in general, it is well understood in the case of surfaces. For surfaces of revolution diffeomorphic to the sphere, a refinement of this problem was introduced by Borzellino, Jordan-Squire, Petrics, and Sullivan. In this article, we quantify their result by counting distinct geodesics of bounded length. In addition, we reframe these results to obtain a couple of characterizations of the round two-sphere.
\end{abstract}
\maketitle

All closed Riemannian manifolds contain a closed geodesic. If the manifold is not simply connected, any length-minimizing representative of a nontrivial homotopy class is a closed geodesic. In the simply connected case, this is already a nontrivial result.

A more difficult question is whether there exist infinitely many closed geodesics. To avoid over-counting, one considers two geodesics \textit{geometrically distinct} if their images are distinct. This brings us to the well known question of whether there exist infinitely many geometrically distinct closed geodesics. In this article, we restrict our attention to surfaces, but we refer the reader to Berger \cite[Chapter XII]{BergerGeometryRevealed} and references therein for a survey on this problem.

For surfaces with genus $g \geq 1$, one uses the infinitude of the  fundamental group and a length minimization argument to construct infinitely many geometrically distinct closed geodesics. For the torus, it follows that the number $N(\ell)$ of such geodesics of length at most $\ell$ grows quadratically in $\ell$. Refining this argument for $g \geq 2$, Katok proved that the number $N(\ell)$ is asymptotically at least $c e^\ell/\ell$ for some constant $c > 0$. 

In the remaining case, when the surface is the sphere, this question was only answered affirmatively in the 1990s by Bangert and Franks \cite{Bangert93,Franks92}. Hingston then proved a quantified version of this result (see \cite{Hingston93}): Given any metric on $\s^2$, the number of geometrically distinct closed geodesics of length at most $\ell$ is asymptotically at least $c \ell /\log \ell$ for some constant $c>0$.

In this article, we consider refinements of these results. As motivation, consider a surface of revolution. Each profile curve connecting the poles extends to a closed geodesic. In particular, the results of Bangert--Franks and Hingston are trivial in this setting. On the other hand, all of these geodesics are in some sense the same. This motivates the following definition: For a closed Riemannian manifold $M$, we say that two geodesics on $M$ are \textit{strongly geometrically distinct} if there is no isometry taking the image of one to the image of the other.

For metrics with finite isometry group, one has immediate analogues of the results above. For metrics with infinite symmetry, it is unclear whether there exist infinitely many strongly geometrically distinct geodesics. For example, the constant curvature metric on $\s^2$ has a unique closed geodesic in this sense. In \cite{BorzellinoEtAl07}, Borzellino et al. prove that all surfaces of revolution diffeomorphic to $\s^2$, except for the round sphere, have infinitely many strongly geometrically distinct geodesics. Our main result is a quantification of this result, as well as a straightforward observation that it extends to all closed, orientable surfaces with continuous (equivalently infinite) symmetry.

\begin{maintheorem}\label{thm:RoundCharacterization}
Let $M$ be an orientable, compact surface with infinite isometry group. Let $N(\ell)$ denote the number of strongly geometrically distinct closed geodesics on $M$ of length less than or equal to $\ell$. One of the following occurs:
	\begin{enumerate}
	\item $M$ is isometric to a round sphere, and $N(\ell) = 1$ for all sufficiently large $\ell > 0$.
	\item There is a constant $c > 0$ such that $N(\ell) \geq c \ell^2$ for all sufficiently large $\ell > 0$.
	\end{enumerate}
\end{maintheorem}

It is well known that $M$ can have infinite isometry group only if $M$ is diffeomorphic to $\s^2$ or the torus $T^2$. In the latter case, a simple extension of a standard argument shows the main theorem holds. However the argument we provide for $\s^2$ carries over with little effort to the case of $T^2$, so we include it in Section \ref{sec:proofT2} for completeness.

Consider now a metric on $\s^2$ with infinite isometry group. The metric takes the form $ds^2 + h(s)^2 d\theta^2$ and one can check that the arguments in Borzellino et al. for a surface of revolution carry over to this slightly more general case to show that infinitely many strongly geometrically distinct closed geodesics exist, i.e., $\lim_{\ell \to \infty} N(\ell) = \infty$. In Section \ref{sec:proofS2}, we summarize their argument and supplement it where needed to prove the claimed lower bound on the growth rate of $N(\ell)$.

Before starting the proof, we point out that this theorem, combined with the work of Hingston and Katok, immediately implies the following:

\begin{nonumbercorollary}\label{cor:RoundCharacterization}
Let $M$ be an orientable, compact surface. Either $M$ is isometric to a round sphere and $N(\ell) = 1$ for all sufficiently large $\ell > 0$, or there exists a constant $c > 0$ such that $N(\ell) \geq c \ell / \log \ell$ for all sufficiently large $\ell>0$.
\end{nonumbercorollary}

\subsection*{Acknowledgements}
This project began as part of the Summer Undergraduate Research Fellowship program in the College of Creative Studies at UCSB. The second author is grateful for the support provided by this program. The first author is partially supported by NSF grants DMS-1045292 and DMS-1404670. Both authors would like to thank Wolfgang Ziller for helpful comments in the preparation of this article.

\bigskip\section{Proof of main theorem for the sphere}\label{sec:proofS2}
\bigskip

Suppose $M$ is a Riemannian manifold diffeomorphic to $\s^2$ with infinite isometry group. Since the isometry group of $M$ is a compact Lie group, this implies the existence of an isometric circle action on $M$. Let $\{p,q\} \subseteq M$ denote the fixed point set of this circle action, and choose a minimal geodesic $c$ from $p$ to $q$. By rescaling the metric if necessary, assume that $c$ is defined on $[0,\pi]$ and that $c(0) = p$ and $c(\pi) = q$. There exists a smooth function $h:(0,\pi) \to (0, \infty)$ and an isometric covering map 
	\begin{eqnarray*}
	\sigma\colon \of{(0,\pi) \times \R, ds^2 + h(s)^2 d\theta^2} &\longrightarrow& M \setminus \{p,q\}\\
		(s,\theta)				&\mapsto		& e^{i\theta}\cdot c(s),
	\end{eqnarray*}
where the dot denotes the action of the circle element $e^{i\theta}$ on $c(s)$. Since $M$ is smooth at $p = c(0)$ and $q = c(\pi)$, we conclude that the extended function $h:[0,\pi] \to \R$ satisfies $h(0) = h(\pi) = 0$ and $h'(0) = -h'(\pi) = 1$ (see \cite[Section 1.3.4]{Petersen06}). The strategy now is to follow the proof in Borzellino et al. \cite{BorzellinoEtAl07}, which covers the case of a surface of revolution. Although we are considering a more general class of surfaces, their arguments extend without any changes to our situation. We summarize the argument here since our strategy is simply to supplement it, as needed, in order to prove the main theorem.

In the coordinates induced by $\sigma$, the geodesic equations are
	\begin{eqnarray*}
	s''(t)		&=& h(s(t)) h'(s(t)) \theta'(t)^2,\\
	\theta''(t)	&=& - 2 \frac{h'(s(t))}{h(s(t))} s'(t) \theta'(t).
	\end{eqnarray*}
The meridians, $\gamma(t) = \sigma(t,\theta_0)$, satisfy these equations and extend to closed geodesics passing through both poles, $p$ and $q$. Since $\theta_0$ is arbitrary, we have by uniqueness that meridians are the only geodesics that pass through the poles. In the rest of this section, we consider those geodesics that do not pass through the poles. Since $\sigma$ defines an isometric covering map onto $M \setminus\{p,q\}$, we can write a geodesic $\gamma(t)$ as $\sigma(s(t), \theta(t))$ for smooth functions $s:\R \to (0,\pi)$ and $\theta\colon\R\to\R$. For example, the parallels given by $\gamma(t) = \sigma(s_0, t/h(s_0))$ are closed geodesics provided that $h'(s_0) = 0$.

An important consequence of the geodesic equations is Clairaut's relation. This states that, for each non-meridian geodesic $\gamma$, there exists a positive constant $c_\gamma$ such that
	\[h(s(t)) \cos \alpha(t) = c_\ga,\]
where $\alpha(t)$ is the angle between $\gamma'(t)$ and the coordinate vector field $\sigma_\theta$ at $\gamma(t)$. Since the cosine function is bounded, $h(s(t))$ cannot go to zero, hence any non-meridian curve has its $s$--coordinate bounded by some interval
	\[[s_0(\ga), s_1(\ga)] = [\inf s(t), \sup s(t)] \subseteq (0,\pi).\]
Further analysis shows the following.

\begin{lemma}[Clairaut]\label{lem:Clairaut}
For $a \in (0,\pi)$, let $\gamma_a$ denote a unit-speed geodesic starting with $s$--coordinate $a$ and with initial direction $\gamma'(0)$ in the $\theta$-direction. Exactly one of the following occurs:
	\begin{description}
	\item[1. parallel] $h'(a) = 0$, and $s(t) = a$ for all $t$.
	\item[2. asymptotic] $h'(a) > 0$ (resp. $<0$) and there exists $b = b(a) > a$ (resp. $<a$) such that $h'(b) = 0$ and $s(t) \to b$ as $t \to \infty$.
	\item[3. oscillating] $h'(a) > 0$ (resp. $<0$) and there exists $b = b(a) > a$ (resp. $<a$) such that $h'(b) < 0$ (resp. $>0$) and $s(t)$ oscillates between $a$ and $b$, achieving these extremal values at integral multiples of some time, denoted $T(a)$.
	\end{description}
\end{lemma}

According to this result, we refer to the parameter $a \in (0,\pi)$ as parallel, asymptotic, or oscillating. Following \cite[Proposition 3.1]{BorzellinoEtAl07}, we let $U \subseteq (0,\pi)$ denote the subset consisting of oscillating $a \in (0,\pi)$ for which $h'(a) > 0$ and $h'(b(a)) < 0$, where $b(a) = \inf \{b > a \st h(b) = h(a)\}$. Geometrically, the $s$--coordinate of $\gamma_a$ oscillates between $a$ and $b(a)$. It follows that $U \subseteq (0,\pi)$ is an open set and that the function $a \mapsto b(a)$ on $U$ is smooth. Indeed, this function is given by $h$ composed with a local inverse of $h$, and so it is smooth by the inverse function theorem.

For each $a \in U$, write $\gamma_a(t) = \sigma(s(t), \theta(t))$ and define
	\begin{eqnarray*}
	R(a)	&=&	2 \int_0^{T(a)} \theta'(t) dt,\\
	L(a)	&=& 2 \int_0^{T(a)} \sqrt{s'(t)^2 + h(s(t))^2 \theta'(t)^2} dt.
	\end{eqnarray*}
This defines two functions $R : U \to \R$ and $L: U \to \R$. The geometric interpretation of these functions is as follows. The quantity $2T(a)$ denotes the time required for a geodesic starting at $s = a$ and parallel to $\sigma_\theta$ to have its $s$--coordinate go to $b(a)$ and back to $a$. We call this a ``full trip''. It then follows by symmetry that $R(a)$ and $L(a)$ denotes the total rotation and length of the geodesic on a full trip. In \cite{BorzellinoEtAl07}, the authors prove that $R(a)$ is a continuous function of $a$. For our purposes, we also need that $L(a)$ is continuous.

\begin{lemma}
The functions $L, R:U \to \R$ are continuous.
\end{lemma}

\begin{proof}
The proofs for $R$ and $L$ are similar, so we only prove it for $L$. Fix $a \in U$. Choose an non-trivial interval $[a_1,a_2] \subseteq U$ containing $a$  on which $h' \geq c_1 > 0$. We prove now that $L$ is continuous on $[a_1, a_2]$.

The geodesic equations imply that 
	$L(a) = 2 \int_a^{b(a)} l(a,s) ds$, where $l(a,s)$ is given by $h(s)/\sqrt{h(s)^2 - h(a)^2}$. 
This integral is improper at both endpoints, so we proceed by proving the following two claims:
	\begin{enumerate}
	\item For all sufficiently small $\de > 0$, $L_\de(a) = 2 \int_{a+\de}^{b(a) - \de} l(a,s)ds$ is smooth.
	\item The functions $L_\de$ converge uniformly to $L$ on $[a_1,a_2]$.
	\end{enumerate}
The first claim follows from the Leibniz integral rule since $l(a,s)$ is a smooth function on the set 
$\{(a,s)| a\in[a_1, a_2], a+\delta\leq s\leq b(a)-\delta\}$. 
To prove the second claim, it suffices to prove that
	$\int_{a}^{a+\de} l(a,s)ds \to 0$ and  $\int_{b(a)-\de}^{b(a)} l(a,s) ds \to 0$
uniformly in $a \in [a_1, a_2]$ as $\de$ goes to $0$. These claims are proven similarly, so we only prove the first. The second only requires the additional fact that $b(a)$ depends smoothly on $a$.

Observe that $l(a,s)$ is non-negative and bounded above as
	\[l(a,s) = \frac{h(s)}{\sqrt{h(s)^2 - h(a)^2}} \leq \frac{1}{2c_1}   \frac{2 h(s) h'(s)}{\sqrt{h(s)^2 - h(a)^2}}.\]
Integrating this expression and applying the change of variables $y = h(s)^2 - h(a)^2$, we conclude that 
	\[\int_a^{a+\de} l(a,s) ds
	\leq \frac{1}{2 c_1} \int_{0}^{h(a+\de)^2 - h(a)^2} \frac{dy}{\sqrt{y}}
	= \frac{\sqrt{h(a+\de)^2 - h(a)^2}}{c_1}.\]
Since $h$ is smooth and hence uniformly continuous on $[0,\pi]$, this last quantity converges to $0$ uniformly in $a$ as $\de \to 0$. This completes the proof.
\end{proof}

We proceed to the proof of the main theorem, that the number $N(\ell)$ of strongly geometrically distinct closed geodesics grows quadratically in $\ell$. The idea is to show, for all large $\ell > 0$, that a large number of values of $a$ exist such that $a \in U$, $R(a) = 2\pi \frac{p}{q}$ for some rational $\frac p q$, and $L(a) \leq \ell/q$. These three conditions imply that any choice of $\gamma_a$ as in Lemma \ref{lem:Clairaut} is oscillating, closes up after $q$ ``full trips'', and is a closed geodesic with length at most $\ell$. 

Since $M$ is diffeomorphic to $\s^2$, the isometry group cannot be two-dimensional. We may assume without loss of generality that the isometry group is one-dimensional. By compactness, it is simply a finite cover of a circle. In particular, for each value of $a$ as above, at most finitely many other such values result in geodesics that are not strongly geometrically distinct from $\gamma_a$. This issue results in a multiplicative factor (equal to the finite number of components in the isometry group) in our estimates. Since the main theorem involves an unknown multiplicative constant anyway, we simply assume, without loss of generality, that the isometry group equals the circle.

The proof is carried out in three cases, which are based roughly on the setup in \cite{BorzellinoEtAl07}. One key step is to prove that there exists an asymptotic geodesic if $h$ has more than one critical point. This actually need not be the case. Indeed, a capped cylinder provides a counterexample, since every critical point is a local maximum and hence not a limiting value of an asymptotic geodesic. This problem is easy to fix however. Indeed, we break the proof into the following cases:

\begin{lemma}
If $h$ has infinitely many critical points, then $N(\ell) = \infty$ for all sufficiently large $\ell > 0$.
\end{lemma}

\begin{proof}
If $h'(a) = 0$, then $\gamma_a(t) = \sigma(a, t/h(a))$ is a closed geodesic of length $2\pi h(a)$. Moreover, the image of $\gamma_a$ maps to itself under any isometry, so distinct values of $a$ yield strongly geometrically distinct closed geodesics. The result follows since $h$ is bounded on $[0,\pi]$.
\end{proof}

\begin{lemma}
If $h$ has finitely many critical points, and $R$ is locally constant, then $N(\ell) = \infty$ for all sufficiently large $\ell > 0$.
\end{lemma}

\begin{proof}
In this case, the argument in \cite[Corollaries 4.4 and 4.5]{BorzellinoEtAl07} is valid since the critical points are isolated. Indeed, first suppose that $h$ has more than one critical point. The arguments there show that $M$ has an asymptotic geodesic and hence that $R$ is unbounded on $U$. However Lemma \ref{lem:Clairaut} and the assumptions of this lemma imply that $R$ takes on only finitely many values, so this is a contradiction. Assume instead that $h$ has a unique critical point, $s_0$. It follows as in \cite[Corollary 5.4]{BorzellinoEtAl07} that $U = (0,s_0)$ and that $R(a) = \lim_{a' \to 0} R(a') = 2 \pi$ for all $a \in (0, s_0)$. But $L$ is continuous on $(0,s_0)$ and hence on $[\frac{s_0}{3},\frac{s_0}{2}]$, so there exist infinitely many strongly geometrically distinct closed geodesics of length at most $L_0$ where $L_0 = \max\{L(s) \st s \in [\frac{s_0}{3}, \frac{s_0}{2}]\} < \infty$. 
\end{proof}

\begin{lemma}
If $h$ has finitely many critical points and $R$ is not locally constant, then there exists a constant $c > 0$ such that $N(\ell) \geq c\ell^2$ for all sufficiently large $\ell > 0$.
\end{lemma}

\begin{proof}
Choose a closed interval $I' \subseteq U$ that is mapped by $R$ to some non-trivial interval $I \subseteq \R$. Let $2\pi \frac p q \in I$. Each $a \in U$ that is mapped by $R$ to $2 \pi \frac p q$ corresponds to a closed geodesic of length $q L(a)$.  Since $L$ is continuous on $I'$, this length is at most $q L_0$, where $L_0$ is the maximum value of $L$ on $I'$. This length is at most $\ell$ if and only if $q \leq \floor{\ell/L_0}$. To estimate $N(\ell)$ from below, it suffices to count the number of rationals $\frac p q \in \frac{1}{2\pi}I$ with $q \leq \floor{\ell/L_0}$. By Lemma \ref{lem:counting} below, there is a constant $c'$ such that the number of such rationals is at least $c' \of{\floor{\ell/L_0}}^2$ for all sufficiently large $\ell$. Taking $c = \frac 1 2 c' /L_0^{2}$, we conclude that $N(\ell) \geq c \ell^2$ for all sufficiently large $\ell > 0$.
\end{proof}

As indicated in the previous proof, it suffices to prove the following counting lemma.

\begin{lemma}\label{lem:counting}
Inside any connected, non-trivial interval $I \subseteq \R$, there exist constants $c > 0$ and $n_0 \in \N$ such that for all $n \geq n_0$, there are at least $c n^2$ rational numbers in $I$ with denominator at most $n$.
\end{lemma}

\begin{proof}
The proof uses Farey fractions. Let $F_n$ denote the set of rationals $a/b$ written in reduced form such that $0 \leq a \leq b \leq n$. It is easy to see that the number of elements in $F_n$ satisfies
	\[|F_n| =1 + \sum_{k=1}^n \phi(k),\]
where $\phi(k)$ is the Euler totient function, given by the number of integers $1 \leq i \leq k$ coprime to $k$. According to Walfisz \cite{Walfisz63}, 
	\[\sum_{k=1}^n \phi(k) = \frac{3}{\pi^2} n^2 + \mathrm{O}\of{n \of{\log n}^{2/3} \of{\log\log n}^{4/3}}.\]
In particular, it follows that constants $c_1 > 0$ and $n_0 > 0$ exist such that $|F_n| > c_1 n^2$ for all $n \geq n_0$.

The idea now is to inject $F_n$ into $I$ in a controlled way. First, it is clear that the conclusion of the lemma holds for $I$ if and only if it holds for $ \{1 + i \st i \in I\}$. Hence, we assume without loss of generality that $I \not\subseteq (-\infty, 0]$. Choose positive integers $a$ and $b$ such that $I$ contains the interval $\left[\frac a b , \frac{a+1}{b}\right]$. Set $c = \frac 1 2 \pfrac{c_1}{b^2}$, and choose $n_0 \geq n_1$ such that $\floor{n/b} \geq n_1$ and $c_1\of{\frac n b - 1}^2 > c n^2$ for all $n \geq n_0$. We claim that $n \geq n_0$ implies that the number of rationals $x \in I$ with denominator at most $n$ is at least $c n^2$.

To do this, consider the injection $F_{\floor{n/b}} \to I$ given by $x \mapsto \frac{a+x}{b}$. Note that the rationals in the image of this map have denominator at most $n$. Hence the total number of rationals in $I$ with denominator at most $n$ is at least the order of $F_{\floor{n/b}}$. For all $n \geq n_0$, this order is at least $c_1 \of{\floor{n/b}}^2$, which in turn is greater than $c n^2$.
\end{proof}

This completes the proof of the main theorem in the case where $M$ is a sphere.

\bigskip
\section{Proof of main theorem for the torus}\label{sec:proofT2}
\bigskip

Assume now that $M$ is diffeomorphic to the torus and has infinite isometry group. In this case, there exists an isometric covering map from
	\[\sigma:(\R \times \R, ds^2 + h(s)^2 d\theta^2) \to M,\]
where $h:\R \to \R$ is some smooth, positive, and periodic function on $\R$. To fix notation, we perform a global scaling so that the period is $\pi$.

As with the case where $M$ is diffeomorphic to $\s^2$, we obtain the same geodesic equations and Clairaut relation. However, Lemma \ref{lem:Clairaut} does not hold since it is possible for geodesics to have the property $|s(t)| \to \infty$ as $t \to \infty$. Indeed, this is the case for meridians. As a substitute, we make the following easy observation.

\begin{lemma}
The $\pi$--periodic function $h:\R \to \R$ has at least one of the two following properties:
	\begin{enumerate}
	\item (non-isolated case) There exist infinitely many critical points in $(0,\pi)$.
	\item (asymptotic case) There exists an isolated local minimum at some $s_0 \in \R$.
	\end{enumerate}
\end{lemma}

In the first case of the lemma, it follows that $N(\ell) = \infty$ for all $\ell \geq 2 \pi \max(h)$. In the second case, it follows as in the case where $M$ is a sphere that an asymptotic geodesic exists, that the rotation function $R(a)$ is unbounded and hence not locally constant, and therefore that $N(\ell) \geq c \ell^2$ asymptotically in $\ell$ for some constant $c>0$.

%\bibliographystyle{plain}
%\bibliography{myrefs}

\end{document}